\documentclass[reqno, 11pt]{amsart}
\usepackage{amsmath,mathtools}
 \usepackage{amssymb}
\usepackage{amsthm}
\usepackage{amsmath}
\usepackage{times}
\usepackage{latexsym}
\usepackage[mathscr]{eucal}
\usepackage{graphicx}
\usepackage{float}

\numberwithin{equation}{section}
 
  \newtheorem{theorem}{Theorem}[section]
  
  \newtheorem{lemma}[theorem]{Lemma}
  \newtheorem{corollary}[theorem]{Corollary}

  \newtheorem{remark}[theorem]{Remark}
  
  \newtheorem{definition}[theorem]{Definition}

\title[Geometry of null submanifolds]{On complex contact manifolds and their null submanifolds}
\author[Samuel Ssekajja]{Samuel Ssekajja*}
\newcommand{\acr}{\newline\indent}
\address{\llap{*\,} School of Mathematics\acr
 University of the Witwatersrand\acr
 Private Bag 3, Wits 2050\acr
South Africa}
\email{samuel.ssekajja@wits.ac.za, ssekajja.samuel.buwaga@aims-senegal.org}
\thanks{}

\author[Ange Maloko]{Ange Maloko**}
\address{\llap{**\,} Departement des Sciences exactes, Ecole Normale Sup\'erieure \acr
 Universit\'e Marien Ngouabi\acr
 Brazzaville, Republique du Congo}
\email{angemaloko@gmail.com, ange.malokomavanga@umng.cg}
\thanks{}


\begin{document}
\begin{abstract}
In the present paper, we study the geometry of certain classes of null submanifolds of indefinite complex contact manifolds. In particular, we show that quaternion null submanifolds are always totally geodesic. We also present the geometry of distributions on screen real and screen transversal anti-invariant submanifolds.
\end{abstract}
\maketitle
\section{Introduction}

 In the year 1996, K. L. Duggal and A. Bejancu published their  book \cite{db} on the null  geometry of submanifolds which filled an important missing part in the general theory of submanifolds. The book was later updated by K. L. Duggal and B. Sahin in \cite{ds2}, by collecting most of the new discoveries in the area since the first publiccation. In Chapter 8 of \cite{ds2}, the authors introduces the geometry of null submanifolds of indefinite quaternion Kaehler manifolds. The authors studied the geometry of real null hypersurfaces, the structure of null submanifolds, both, of indefinite quaternion Kaehler manifolds and show that a quaternion null submanifold is always totally geodesic. This result implies that the study of null submanifolds, other than quaternion null submanifolds, is interesting. Then, they  dealt with the geometry of screen real submanifolds in detail. As a generalization of real null hypersurfaces of quaternion Kaehler manifolds, they introduced $QR$-null submanifolds. Furthermaore, they show that the class of $QR$-null submanifolds does not include quaternion null submanifolds and screen real submanifolds. They also  introduced and studied the geometry of screen $QR$-null and screen $CR$-null submanifolds as generalizations of quaternion  null submanifolds and screen real submanifolds, and provided examples for each class of null submanifolds of indefinite quaternion Kaehler manifolds. Based on these two books, many researchers have investigated the geometry of null subspaces of semi-Riemannian manifolds.

 On the other hand, in about the same time as in the book \cite{db}, D. N. Kupeli \cite{kup} introduced the theory of null geometry in a relatively different way. The main tool in his approach was the consideration of a factor bundle which is isomorphic to the screen distribution used by the authors in \cite{db}. Despite all the above contributions, we remark that the null geometry of submanifolds of indefinite Sasakian 3-structure manifolds, as well as indefinite complex contact manifold have not yet been studied.

 In \cite{bla1}, the geometry of complex contact manifolds in the Riemannian sense is done in which the foundations on such manifiolds is given, from structures to their curvatures. In \cite{ssekajja}, the geometry of null real hypersurfaces of indefinite complex contact manifold was introduced and many interesting results were proved.  The main objective of this paper is to study  the geometry of $r$-null real submanifolds of indefinite complex contact manifolds. We present quaternion, screen real and screen transversal anti-invariant submanifolds. Several characterisation results are proved. The paper is arranged as follows; In Section \ref{pre}, we quote some basic notions on complex contact manifolds as well as null submanifolds needed in the rest the paper. In Section \ref{mainn}, we study quaternion submanifolds. We prove that these submanifolds are totally geodesic. In Section \ref{main22}, we introduce screen real submanifolds, and in Section \ref{main33} we study screen transversal anti-invariant submanifolds.

\section{Preliminaries} \label{pre}

Let $(\overline{M},\overline{g})$ be an $(m + n)$-dimensional semi-Riemannian manifold of constant index $\nu$, $1\le \nu\le m+n$ and $M$ be a submanifold of $\overline{M}$ of codimension $n$. We assume that both $m$ and $n$ are $\ge 1$. At a point $p\in M$, we define the orthogonal complement $T_{p} M^{\perp}$ of the tangent space $T_{p} M$ by $T_{p} M^{\perp} = \{X\in\Gamma(T_{p} \overline{M}): \overline{g}(X, Y)=0,\; \forall \, Y\in\Gamma(T_{p} M)\}
 $. We put $\mathrm{Rad} \, T_{p} M = \mathrm{Rad}\, T_{p} M^{\perp} = T_{p} M \cap T_{p} M^{\perp}$. The submanifold $M$ of $\overline{M}$ is said to be $r$-null submanifold (one supposes that the index of $\overline{M}$ is $\nu \ge r$), if the mapping $\mathrm{Rad} \, T M: p\in M \longrightarrow\mathrm{Rad}\, T_{p} M $ defines a smooth distribution on $M$ of rank $r > 0$. We call $\mathrm{Rad}\,T M$ the radical distribution on $M$. In the sequel, an $r$-null submanifold will simply be called a \textit{null submanifold} and $g$ is \textit{null metric}, unless we need to specify $r$.

 Let $S(T M)$ be a screen distribution which is a semi-Riemannian complementary distribution of $\mathrm{Rad}\,T M$ in $T M$, that is,
 \begin{align}\label{mas2}
 	TM=\mathrm{Rad}\,TM \perp S(TM).
 \end{align}
 Choose a screen transversal bundle $S(TM^\perp)$, which is semi-Riemannian and complementary to $\mathrm{Rad}\, TM$ in $TM^\perp$. Since, for any local basis $\{\xi_i \}$ of  $\mathrm{Rad}\,TM$, there exists a local null frame $\{N_i\}$ of sections with values in the orthogonal complement of $S(T M^\perp)$ in $S(T M )^\perp$  such that $g(\xi_i , N_j ) = \delta_{ij}$, it follows that there exists a null transversal vector bundle $l\mathrm{tr}(TM)$ locally spanned by $\{N_i\}$ \cite{db}. Let $\mathrm{tr}(TM)$ be complementary (but not orthogonal) vector bundle to $TM$ in $T\overline{M}$. Then,
\begin{align}
           &\mathrm{tr}(TM)=l\mathrm{tr}(TM)\perp S(TM^\perp),\label{eq08}\\
  T\overline{M}= & S(TM)\perp S(TM^\perp)\perp\{\mathrm{Rad}\, TM\oplus l\mathrm{tr}(TM)\}\label{eq04} .
\end{align}
Note that the distribution $S(TM)$ is not unique, and is canonically isomorphic to the factor vector bundle $TM/ \mathrm{Rad}\, TM$  \cite{db}. We say that a null submanifold $M$ of $\overline{M}$ is (1) $r$-null if $1\leq r< min\{m,n\}$; (2) co-isotropic if $1\leq r=n<m$,  $S(TM^\perp)=\{0\}$; (3) isotropic if $1\leq r=m<n$,  $S(TM)=\{0\}$; (4) totally lightlike if $r=n=m$,  $S(TM)=S(TM^\perp)=\{0\}$. The Gauss and Weingarten formulae are given by
\begin{align}
 \overline{\nabla}_{X}Y&=\nabla_{X}Y+h(X,Y),\;\;\forall\,X,Y\in\Gamma(TM),\nonumber\\
 \overline{\nabla}_{X}L=-&A_{L}X+\nabla_{X}^tL,\;\;\forall\,X\in\Gamma(TM),\;\;L\in\Gamma(\mathrm{tr}\,T M),\nonumber
\end{align}
where $\{\nabla_{X}Y,A_{L}X\}$ and $\{h(X,Y),\nabla_{X}^tL\}$ belong to $\Gamma(TM)$ and $\Gamma(\mathrm{tr}(T M))$ respectively. Further,  $\nabla$ and  $\nabla^{t}$ are linear connections on $M$ and $\mathrm{tr}(T M)$, respectively. The second fundamental form $h$ is a symmetric $\mathscr{F}(M)$-bilinear form on $\Gamma(T M)$ with values in $\Gamma(\mathrm{tr}(T M))$ and the shape operator $A_{L}$ is a linear endomorphism of $\Gamma(T M )$. Then we have \cite{db}
\begin{align}
 \overline{\nabla}_{X}Y&=\nabla_{X}Y+h^{l}(X,Y)+h^{s}(X,Y),\label{mas10}\\
 \overline{\nabla}_{X}N&=-A_{N}X+\nabla_{X}^{l}N+D^{s}(X,N),\label{mas11}\\
 \overline{\nabla}_{X}W&=-A_{W}X+\nabla_{X}^{s}W+D^{l}(X,W),\label{mas12}
 \end{align}
for all  $X,Y\in\Gamma(TM)$, $N\in\Gamma(l\mathrm{tr}(T M))$ and $W\in\Gamma(S(TM^{\perp}))$. Here, $A_{N}$ and $A_{w}$ are the shape operators of $M$.

Denote the projection of $TM$ on $S(TM)$ by $P$. Then, by using (\ref{mas2}), (\ref{mas10})-(\ref{mas12}) and a metric connection $\overline{\nabla}$, we obtain
 \begin{equation}\label{mas13}
  \overline{g}(h^{s}(X,Y), W) + \overline{g}(Y,D^{l}(X,W)) = g(A_{W} X,Y),
 \end{equation}
 \begin{equation}\label{mas14}
  \nabla_{X}PY = \nabla^{*}_{X}PY + h^{*}(X,PY),
 \end{equation}
 \begin{equation}\label{mas15}
  \nabla_{X}\xi =-A^{ * }_{\xi}X +\nabla^{* t}_{X} \xi,
 \end{equation}
 \begin{equation}\nonumber
  \overline{g}(h^{l}(X, PY), \xi) = g(A^{*}_{\xi}X, PY),
 \end{equation}
where $X,Y\in\Gamma(TM)$, $\xi\in\Gamma(\mathrm{Rad}\,T M)$ and $W\in\Gamma(S(TM^{\perp}))$. Here, $A^{*}_{\xi}$ is the shape operator of $S(TM)$. In general, the induced connection $\nabla$ on $M$ is not a metric connection. Since $\overline{\nabla}$ is a metric connection, by using (\ref{mas10}) we get
\begin{equation}\label{mas17}
 (\nabla_{X} g)(Y,Z) = \overline{g}(h^{l}(X,Y),Z) + \overline{g}(h^{l}(X,Z),Y),
\end{equation}
for all $X,Y\in \Gamma(TM)$. However, it is important to note that $\nabla^{*}$ is a metric connection on $S(TM)$. We will need the following result in this paper.

\begin{theorem}[\cite{ds2}]\label{th1}
	Let $M$ be an $r$-null submanifold with $r< \min\{m,n\}$ or a coisotropic submanifold of $\overline{M}$. Then the induced linear connection $\nabla$ on $M$ is a metric connection if and only if one of the following conditions is fulfilled:
	\begin{enumerate}
		\item $A_{\xi}^{*}$ vanish on $\Gamma(TM)$ for any $\xi\in \Gamma(\mathrm{Rad}\,TM)$.	
		\item  $\mathrm{Rad}\,TM$ is a Killing distribution.
		\item $\mathrm{Rad}\,TM$ is a parallel distribution with respect to $\nabla$.
\end{enumerate}
\end{theorem}
A {\it complex contact manifold} is a complex manifold, $\overline{M}$, of odd complex dimension $(2n+1)$ together with an open covering $\{\mathcal{O}_{i}\}$ by coordinate neighbourhoods such that: (1) On each $\mathcal{O}_{i}$ there is a holomorphic 1-form $\theta_{i}$ such that $\theta_{i}\wedge (d\theta_{i})^{n}\ne 0$. (2)  On $\mathcal{O}_{i}\cup \mathcal{O}_{j} \ne \emptyset$ there is a non-vanishing holomorphic function $f_{\alpha\beta}$ such that $\theta_{i}=f_{\ij}\theta_{j}$ (see \cite{bla1,bla2} for more details). Furthermore, the subspaces $\{X\in T_{m}\mathcal{O}_{i}: \theta_{i}(X)=0\}$ defines a non-integrable holomorphic subbundle $\mathcal{H}$ of complex dimension $2n$ called the {\it complex contact subbundle} or {\it horizontal subbundle}. The quotient $T\overline{M}/\mathcal{H}$ is a complex line bundle over $\overline{M}$ \cite[p. 49]{bla2}. Some well-known examples of complex contact metric manifolds include the complex Heisenberg group $H_{\mathbb{C}}$ and the odd-dimensional complex projective space, see \cite{bla1, bla2} for more details on these manifolds. Define a local section $U$ of $T\overline{M}$, i.e., a section of $T\mathcal{O}$,  by $du(U,X)=0$,  for every $X\in \mathcal{H}$, $u(U)=1$ and $v(U)=0$. Such local sections then define a global subbundle $\mathcal{V}$ by $\mathcal{V}|_{\mathcal{O}}=\mathrm{Span}\{U,JU\}$. Then, we have $T\overline{M}=\mathcal{H}\perp \mathcal{V}$ and we denote the projection map to $\mathcal{H}$ by $p:T\overline{M}\longrightarrow \mathcal{H}$. The subbundle $\mathcal{V}$ is called the {\it vertical subbundle} or {\it characteristic subbundle}. On the other hand, if $\overline{M}$ is a complex manifold with almost complex structure $J$, Hermitian metric $\overline{g}$ and open covering by coordinate neighbourhoods $\{\mathcal{O}_{i}\}$, $\overline{M}$ is called a {\it complex almost contact metric manifold} if it satisfies the following two conditions: (1) On each $\mathcal{O}_{i}$, there exists 1-forms $u_{i}$, $v_{i}=u_{i}\circ J$, with orthogonal dual vector fields $U_{i}$ and $V_{i}=-JU_{i}$, and (1,1)-tensor fields $G_{i}$ and $H_{i}=G_{i}J$ such that
\begin{align}
	&\;\;\;\;\;\;\;\;H^{2}_{i}=G^{2}_{i}=-I+u_{i}\otimes U_{i}+v_{i}\otimes V_{i},\label{cm1}\\
	&\overline{g}(G_{i}X,Y)=-\overline{g}(X,G_{i}Y),\;\;\;\; \overline{g}(U_{i},X)=u_{i}(X),\label{cm2}\\
	&\;\;\;\;\;\;\;\;G_{i}J=-JG_{i},\;\;\; G_{i}U=0,\;\;\; u_{i}(U)=1,\label{cm3}
\end{align}
 for all $X,Y\in\Gamma(T\overline{M})$.  (2) On the overlaps $\mathcal{O}_{i}\cap \mathcal{O}_{j}\ne \emptyset$, the above tensors transform as $u_{j}=au_{i}-bv_{i}$, $v_{j}=bu_{i}+av_{i}$, $G_{j}=aG_{i}-bH_{i}$ and  $H_{j}=bG_{i}+aH_{i}$, for some functions $a$, $b$ defined on the overlaps with $a^{2}+b^{2}=1$. It is obvious that $H_{i}$ also anticommutes with $J$ and is skew-symmetric with respect to $\overline{g}$ and that $G_{i}$ and $H_{i}$ annihilate both $U$ and $V$. Furthermore, the local contact form $\theta$ is $u-iv$ to within a nonvanishing complex-valued function multiple (see \cite{bla1}). Moreover, given a complex contact manifold, a complex almost contact metric structure can be chosen such that $du(X,Y)=\overline{g}(X,GY)+(\sigma \wedge v)(X,Y)$ and $dv(X,Y)=\overline{g}(X,HY)-(\sigma \wedge u)(X,Y)$, for all $X,Y\in \Gamma(T\overline{M})$, for some 1-form $\sigma$.  In this case we say that $\overline{M}$ has a {\it complex contact metric structure} $(u, v, U, V, G, H, \overline{g})$  \cite{bla1,bla2}.  In this case $\sigma(X)=\overline{g}(\overline{\nabla}_{X}U,V)$, where $\overline{\nabla}$ denotes the Levi-Civita connection on $\overline{M}$. We refer to a complex contact manifold with a complex almost contact metric structure satisfying these conditions as a {\it complex contact metric manifold} \cite{bla1,bla2}.

 Next,  for a complex contact metric structure \cite[p. 237]{bla1} defined local tensor fields $h_{U}$  and $h_{V}$ by $h_{U}=\frac{1}{2}\mathrm{sym}(\pounds_{U}G)\circ p$ and  $h_{V}=\frac{1}{2}\mathrm{sym}(\pounds_{V}H)\circ p$, where $\mathrm{sym}$ denotes the symmetric part; $h_{U}$ anticommutes with $G$, $h_{V}$ anticommutes with $H$, and
 \begin{align}
 	&\overline{\nabla}_{X}U=-GX-Gh_{U}X+\sigma(X)V,\label{cm4}\\
 	\mbox{and}\;\;\;&\overline{\nabla}_{X}V=-HX-Hh_{V}X-\sigma(X)U.\label{cm5}
  \end{align}
In view of (\ref{cm4}) and (\ref{cm5}) one easily sees that the integral surfaces of $\mathcal{V}$ are totally geodesic submanifolds. Furthermore, the associated metric $\overline{g}$ is projectable with respect to the foliation induced by the integrable subbundle $\mathcal{V}$ if and only if $h_{U}$ and $h_{V}$ vanish (see \cite{bla1} for more details).
\begin{lemma}\label{lemma11}
	The tensors $J,G$ and $H$ satisfies the following relations
	\begin{align}
		&JG=-H,\;\;\;JH=-HJ=G,\label{j1}\\
		H&G=-GH=J+u\otimes V-v\otimes U,\label{j2}\\
		GU=H&U=HV=0,\;\;\;\overline{g}(HX,Y)=-\overline{g}(X,HY),\label{j3}
	\end{align}
	for all $X,Y\in \Gamma(T\overline{M})$.
\end{lemma}
\begin{proof}
	Using (\ref{cm3}), we have $JG=-GJ=-H$. On the other hand, $JH=-HJ=-GJ^{2}=G$. This proves (\ref{j1}). Then, in view of (\ref{cm1}) and (\ref{cm3}), we have $HG=GJG=-GGJ=-GH$, and $-GH=-G^{2}J=J-(u\circ J)\otimes U-(v\circ J)\otimes V=J+u\otimes V-v\otimes U$, proving (\ref{j2}). Note that $GU=0$, by (\ref{cm3}). Also, $HU=GJU=-JGU=0$. Furthermore, $HV=GJV=-GJ^{2}U=GU=0$. Finally, for any $X,Y\in \Gamma(T\overline{M})$, we have $\overline{g}(HX,Y)=\overline{g}(GJX,Y)=-\overline{g}(JX,GY)=\overline{g}(X,JGY)=-\overline{g}(X,GJY)=-\overline{g}(X,HY)$, in which we have used (\ref{cm2}) and (\ref{cm3}), which completes the proof.
\end{proof}

Consider the  tensor fields $S$ and $T$ given by
\begin{align}
	S&(X,Y)=[G,G](X,Y)+2\overline{g}(X,GY)U-2\overline{g}(X,HY)V\nonumber\\
	&+2\{v(Y)HX-v(X)HY\}+\sigma(GY)HX-\sigma(GX)HY\nonumber\\
	&\;\;\;\;\;\;\;\;\;\;\;\;\;\;\;\;\;\;\;\;+\sigma(X)GHY-\sigma(Y)GHX,\nonumber
\end{align}
\begin{align}
	T&(X,Y)=[H,H](X,Y)-2\overline{g}(X,GY)U+2\overline{g}(X,HY)V\nonumber\\
	&+2\{u(Y)GX-u(X)GY\}+\sigma(HX)GY-\sigma(HY)GX\nonumber\\
	&\;\;\;\;\;\;\;\;\;\;\;\;\;\;\;\;\;\;\;\;+\sigma(X)GHY-\sigma(Y)GHX,\nonumber
\end{align}
for all $X,Y\in \Gamma(T\overline{M})$. In the above, $[G,G]$ and $[H,H]$ denotes the Nijenhuis tensors of $G$ and $H$, respectively. Then, a complex contact metric structure is {\it normal} \cite[p. 251]{bla1} if  $S(X,Y)=T(X,Y)=0$, for all $X,Y\in \Gamma(\mathcal{H})$ and $S(U,X)=T(V,X)=0$, for all $X\in \Gamma(T\overline{M})$. {\it An important consequence of normality is that $h_{U}=0$, for every $U\in \Gamma(\mathcal{V})$}, see \cite[p. 251]{bla1}. {\it Throughout this paper, we take $\overline{M}$ to be normal}. Moreover, on  a normal complex contact manifold, $\overline{\nabla}J$, $\overline{\nabla}G$ and $\overline{\nabla}H$ satisfies the relations (see \cite[p. 252]{bla1} for more details).
\begin{align}
\overline{g}((\overline{\nabla}_{X}J)Y,Z)&=u(X)\nonumber\{d\sigma(Z,GY)-2\overline{g}(HY,Z)\}\nonumber\\
&\;\;\;\;\;\;\;\;\;\;\;\;\;\;\;\;\;\;+v(X)\{d\sigma(Z,HY)+ \overline{g}(GY,Z)\},\label{m1}\\
\overline{g}((\overline{\nabla}_{X}G)Y,Z)&=\sigma(X)\overline{g}(HY,Z)+v(X)d\sigma(GZ,GY)\nonumber\\
	&\;\;\;\;-2v(X)\overline{g}(HGY,Z)-u(Y)\overline{g}(X,Z) -v(Y)\overline{g}(JY,Z)\nonumber\\
	&\;\;\;\;\;\;\;\;\;\;\;\;\;\;\;\;\;\;\;\;+u(Z)\overline{g}(X,Y)+v(Z)\overline{g}(JX,Y),\label{e20}\\
	\overline{g}((\overline{\nabla}_{X}H)Y,Z)&=-\sigma(X)\overline{g}(GY,Z)-u(X)d\sigma(HZ,HY)\nonumber\\
	&\;\;\;\;-2u(X)\overline{g}(GHY,Z)-v(Y)\overline{g}(X,Z) +u(Y)\overline{g}(JY,Z)\nonumber\\
	&\;\;\;\;\;\;\;\;\;\;\;\;\;\;\;\;\;\;\;\;\;\;+v(Z)\overline{g}(X,Y)-u(Z)\overline{g}(JX,Y),\label{e21}
	\end{align}
for all $X,Y,Z\in \Gamma(T\overline{M})$.

From now on, we assume that $\overline{M}:=(\overline{M},u, v, U, V, G, H, \overline{g})$ is a $(4n+2)$-dimensional indefinite complex contact manifold, where $\overline{g}$ is a semi-Riemannian metric of index $4q$; $0<q<n$.

\section{Quaternion null submanifolds}\label{mainn}

In this section, we consider null submanifolds whose tangent bundles contain the characteristic vector bundle $\mathcal{V}=\mathrm{Span}\{U,V\}$ as their subbundle. Since $\mathcal{V}$ is a spacelike subbundle, it follows that if $\mathcal{V}\subset TM$, then it can not be a subbundle of $\mathrm{Rad}\,TM$ as it is already known to be degenerate. Therefore when $\mathcal{V}\subset TM$,  we shall always mean  $\mathcal{V}\subset S(TM)$. Then, borrowing from \cite[p. 366]{ds2}, we have the following definition

 \begin{definition}
 \rm{
 	A null submanifold $(M,g)$, tangent to $\mathcal{V}$, of an indefinite complex contact manifold $\overline{M}$ will be called a  quaternion null submanifold if both $\mathrm{Rad}\,TM$ and $S(TM)$ are invariant with respect to each tensor $J$, $G$ and $H$.
 	}
 \end{definition}
 In fact, if we let
\begin{align}
	J_{1}:=J,\;\;\;J_{2}:=G\;\;\;\mbox{and}\;\;\; J_{3}:=H,
\end{align}
then $M$ is a quaternion null submanifold if
\begin{align}
	J_{a}\mathrm{Rad}\,TM=\mathrm{Rad}\,TM\;\;\; \mbox{and}\;\;\; J_{a}S(TM)=S(TM),
\end{align}
 for all $a=1,2,3$. It is easy to see that the null transversal bundle $l\mathrm{tr}(TM)$  and the screen transversal bundle $S(TM^{\perp})$ are also invariant with respect to $J_{a}$. Next, assume that $M$ is quaternion null submanifold of a normal indefinite complex contact manifold, and let $X\in \Gamma(\mathrm{Rad}\,TM)$, then (\ref{mas10}) and (\ref{cm4}), and the fact that $\overline{M}$ is normal--leading to $h_{U},h_{V}=0$, implies that
 \begin{align}\label{f1}
 	\nabla_{X}U+h^{l}(X,U)+h^{s}(X,U)=-GX+\sigma(X)V.
 \end{align}
 It follows from (\ref{f1}) that $h^{l}(X,U)=h^{s}(X,U)=0$ and $\nabla_{X}U=-GX+\sigma(X)V$, for all $X\in \Gamma(\mathrm{Rad}\,TM)$. On the other hand, when $X\in
\Gamma(S(TM))$, we also get $h^{l}(X,U)=h^{s}(X,U)=0$ and $\nabla_{X}U=-GX+\sigma(X)V$. Putting all the above together, we have shown that $h^{l}(X,U)=h^{s}(X,U)=0$ and $\nabla_{X}U=-GX+\sigma(X)V$, for all $X\in \Gamma(TM)$. Also, one can easily show, using (\ref{mas10}) and (\ref{cm5}), that  $h^{l}(X,V)=h^{s}(X,V)=0$ and $\nabla_{X}V=-HX-\sigma(X)U$, for any $X\in \Gamma(TM)$. We therefore state the following result:

\begin{lemma}\label{lemmau}
	On any quaternion null submanifold of a normal indefinite complex contact manifolds, the following holds;
\begin{align}
		h^{l}(X,U)&=h^{s}(X,U)=0,\;\;\;\nabla_{X}U=-GX+\sigma(X)V, \label{w1}\\
		 \mbox{and}\;\;\; h^{l}(X,V)&=h^{s}(X,V)=0, \;\; \; \nabla_{X}V=-HX-\sigma(X)U,\label{w2}
\end{align}	
for all $X\in \Gamma(TM)$.
\end{lemma}

\noindent Now, we have the following result.
\begin{theorem}
	On any quaternion null submanifold of a normal indefinite complex contact manifold, the following hold;
	\begin{align}
		&h^{l}(JX,JY)=-h^{l}(X,Y),\;\;\; h^{s}(JX,JY)=-h^{s}(X,Y),\label{a1}\\
		&h^{l}(GX,GY)=-h^{l}(X,Y),\;\;\; h^{s}(GX,GY)=-h^{s}(X,Y),\label{a2}\\
		&h^{l}(HX,HY)=-h^{l}(X,Y),\;\;\; h^{s}(HX,HY)=-h^{s}(X,Y),\label{a3}
	\end{align}
	for all $X,Y\in \Gamma(TM)$.
\end{theorem}
\begin{proof}
	Letting $Z=\xi$ in (\ref{m1}), we get
	\begin{align}\label{w2}
		\overline{g}((\overline{\nabla}_{X}J)Y,\xi)&=u(X)d\sigma(\xi,GY)+v(X)d\sigma(\xi,HY),
	\end{align}
	for all $X,Y\in \Gamma(TM)$. Using the Gauss formula (\ref{mas10}) in (\ref{w2}), gives
	\begin{align}\label{w3}
			\overline{g}(h^{l}(X,JY),\xi)&-\overline{g}(Jh^{l}(X,Y),\xi)\nonumber\\
		&=u(X)d\sigma(\xi,GY)+v(X)d\sigma(\xi,HY),	
\end{align}
for all $X,Y\in \Gamma(TM)$.	 Setting $X=U$  in (\ref{w3}) and using (\ref{w1}), we get $d\sigma(\xi,GY)=0$. On the other hand, setting $X=V$ in (\ref{w3}) and using (\ref{w2}) give us $d\sigma(\xi,HY)=0$, for all $Y\in \Gamma(TM)$. It then follows from (\ref{w3}) that
\begin{align}\label{w4}
	h^{l}(X,JY)=Jh^{l}(X,Y),
\end{align}
for all $X,Y\in \Gamma(TM)$. Then, replacing $X$ with $JX$ in (\ref{w4}), we get
\begin{align}\label{w5}
	h^{l}(JX,JY)=Jh^{l}(JX,Y)=J^{2}h^{l}(X,Y)=-h^{l}(X,Y),
\end{align}
which proves the first relation in (\ref{a1}). Next, letting $Z=W$ in (\ref{m1}) and using Gauss equation (\ref{mas10}), we get
\begin{align}\label{w6}
			\overline{g}(h^{s}(X,JY),W)&-\overline{g}(Jh^{s}(X,Y),W)\nonumber\\
		&=u(X)d\sigma(W,GY)+v(X)d\sigma(W,HY),	
\end{align}
for all $X,Y\in \Gamma(TM)$. Setting $X=U$ and $X=V$ in (\ref{w6}), by turns, and using (\ref{w1}) and (\ref{w2}), we get $\sigma(W,GY)=0$ and $d\sigma(W,HY)=0$, for all $Y\in \Gamma(TM)$. Thus, (\ref{w6}) reduces to
\begin{align}\label{w7}
	h^{s}(X,JY)=Jh^{s}(X,Y).
\end{align}
Replacing $X$ with $JX$ in (\ref{w7}), gives
\begin{align}\label{w8}
	h^{s}(JX,JY)=Jh^{s}(JX,Y)=J^{2}h^{s}(X,Y)=-h^{s}(X,Y),
\end{align}
which proves the second relation in (\ref{a1}).

Turning  to (\ref{e20}), we  let $Z=\xi$ and use Gauss equation (\ref{mas10}) to get
\begin{align}\label{w8}
	\overline{g}(h^{l}(X,GY),\xi)&-\overline{g}(Gh^{l}(X,Y),\xi)=v(X)d\sigma(G\xi,GY),
\end{align}
for all $X,Y\in \Gamma(TM)$. Then, setting $X=V$ in (\ref{w8}) and using (\ref{w2}), we get $d\sigma(G\xi,GY)=0$. It then follows that
\begin{align}\label{w9}
	h^{l}(X,GY)=Gh^{l}(X,Y),
\end{align}
for all $X,Y\in \Gamma(TM)$. Replacing $X$ by $GY$ in (\ref{w9}) and using (\ref{cm1}), gives
\begin{align}\label{w9}
	h^{l}(GX,GY)=Gh^{l}(GX,Y)=G^{2}h^{l}(X,Y)=-h^{l}(X,Y),
\end{align}
which proves the first relation in (\ref{a2}). Next, letting $Z=W$ in (\ref{e20}) and following similar steps to those above, we get
\begin{align}\label{w10}
	h^{s}(GX,GY)=Gh^{s}(GX,Y)=G^{2}h^{l}(X,Y)=-h^{s}(X,Y),
\end{align}
which give us the second relation of (\ref{a2}).

Finally, using (\ref{e21}) and similar steps as above, we get
\begin{align}\label{w11}
	h^{l}(HX,HY)=-h^{l}(X,Y)\;\;\;\mbox{and} \;\;\; h^{s}(HX,HY)=-h^{s}(X,Y),
\end{align}
for all $X,Y\in \Gamma(TM)$, which proves (\ref{a3}), and completing the proof.
\end{proof}
\begin{corollary}\label{cor1}
	Any quaternion null submanifold $(M,g)$ of a normal indefinite complex contact manifold is totally geodesic.
\end{corollary}
\begin{proof}
	By (\ref{j2}) of Lemma \ref{lemma11} and Lemma \ref{lemmau}, we have
	\begin{align}
		h^{l}(JX,JY)&=h^{l}(HGX-u(X)V+v(X)U,HGY-u(Y)V+v(Y)U)\nonumber\\
		&=h^{l}(HGX,HGY).\label{w12}
	\end{align}
	Considering (\ref{a2}), (\ref{a3}) and (\ref{w12}), we derive
	\begin{align}
		h^{l}(JX,JY)&=h^{l}(HGX,HGY)=-h^{l}(GX,GY)=h^{l}(X,Y),\label{w13}
	\end{align}
	for all $X,Y\in \Gamma(TM)$. It then follows from (\ref{w13}) and (\ref{a1}) that $h^{l}(X,Y)=0$, for all $X,Y\in \Gamma(TM)$. On the other hand, using  (\ref{j2}), (\ref{a2})  and (\ref{a3}), we derive
	\begin{align}\label{w14}
		h^{s}(JX,JY)=h^{s}(X,Y).
	\end{align}
	It follows from (\ref{w14}) and (\ref{a1}) that $h^{s}(X,Y)=0$, for all $X,Y\in\Gamma(TM)$. Hence, $M$ is totally geodesic, which completes the proof.
\end{proof}
\begin{remark}
\rm{
\begin{enumerate}
	\item A similar result to Corollary \ref{cor1} was obtained in \cite[Theorem 8.3.1, p. 367]{ds2} for quaternion null submanifold of indefinite quaternion Kaehler manifolds, although in our case arriving at such a result requires a little more work due to the nature of ambient indefinite complex contact manifold.
	\item Corollary \ref{cor1} shows that it would be more interesting to study other types of null submanifolds of indefinite complex contact manifolds.
	\end{enumerate}
}
\end{remark}
\section{Screen real null submanifolds}\label{main22}

In this section, we define screen real null submanifold of an indefinite complex contact manifold. But first, we note that when the characteristic subbundle $\mathcal{V}=Span\{U,V\}$ is a subbundle of the transversal bundle $\mathrm{tr}(TM)(=l\mathrm{tr}(TM)\perp S(TM^{\perp}))$, then it cannot be a subbundle of null transversal bundle $l\mathrm{tr}(TM)$ entirely. This is due to the fact that $l\mathrm{tr}(TM)$ is a null subbundle yet $\mathcal{V}$ is a spacelike subbundle.  Therefore, by $\mathcal{V}\subset \mathrm{tr}(TM)$ we shall always mean $\mathcal{V}\subset S(TM^{\perp})$. This choice has the following consequence.
\begin{lemma}\label{impo}
Let $(M,g)$ be a null submanifold of a normal indefinite complex contact manifold $\overline{M}$. If $\mathcal{V}\subset S(TM^{\perp})$ then $G(TM),H(TM)\subset TM^{\perp}=\mathrm{Rad}\,TM\perp S(TM^{\perp})$.
\end{lemma}
\begin{proof}
	Assume that $\mathcal{V}\subset S(TM^{\perp})$. Then,  from (\ref{cm4}), we have
	\begin{align}\label{u1}
		-\overline{g}(GX,Y)=\overline{g}(\overline{\nabla}_{X}U,Y),
	\end{align}
	for all $X,Y\in \Gamma(TM)$. It then follows from (\ref{u1}) and (\ref{mas10}) that
	\begin{align}\label{u2}
		\overline{g}(GX,Y)=\overline{g}(U,\overline{\nabla}_{X}Y)=\overline{g}(U,h^{s}(X,Y)).
	\end{align}
Note that the left hand side of (\ref{u2}) is skew-symmetric while the right hand side is symmetric. So $\overline{g}(GX,Y)=0$, for all $X,Y\in \Gamma(TM)$. Similarly, $\overline{g}(HX,Y)=0$. Hence, $G(TM)\subset TM^{\perp}$  and  $H(TM)\subset TM^{\perp}$, which completes the proof.
\end{proof}

\begin{definition}\label{def2}
\rm{
	Let $(M,g)$  be a null submanifold of an indefinite complex contact manifold $\overline{M}$, which is transversal to the characteristic subbundle $\mathcal{V}$, that is $\mathcal{V}\subset S(TM^{\perp})$. Then, $M$ is called a screen real null submanifold if $J_{a}\mathrm{Rad}\,TM=\mathrm{Rad}\,TM$ and $J_{a}S(TM)\subset S(TM^{\perp})$, for all $a=1,2,3$.}
\end{definition}
\begin{remark}\label{rema}
\rm {
\begin{enumerate}
	\item It follows immediately from Definition \ref{def2} that the global 1-forms $u$ and $v$ vanishes on $TM$. Moreover, it is easy to see that $l\mathrm{tr}(TM)$ is also invariant with respect to $J_{a}$, for all $a=1,2,3$.
	\item Another important point to note is that $A_{W}$, for all $W\in \Gamma(S(TM^{\perp}))$, is never a  screen valued operator. In fact, using (\ref{mas12}) and (\ref{cm4}), with $W=U$, we have $-GX+\sigma(X)V=-A_{U}X+\nabla_{X}^{s}U+D^{l}(X,U)$, for any $X\in \Gamma(TM)$. It would then mean that $\overline{g}(GX,N)=\overline{g}(A_{U}X,N)$, for all $N\in \Gamma(l\mathrm{tr}(TM))$. Thus, if $A_{U}$  is screen valued, we obtain $\overline{g}(GX,N)=0$. Now, replacing $N$ with $GN$ and using (\ref{m1}), we get $\overline{g}(X,N)=0$, for all $X\in \Gamma(TM)$, which is a contradiction. A similar conclusion is arrived at if one uses (\ref{cm5}).
\end{enumerate}
	}
\end{remark}

\noindent Next, for any $W\in \Gamma(S(TM^{\perp}))$ we have
\begin{align}\label{w20}
	J_{a}W=B_{a}W+C_{a}W,\;\;\forall\, a=1,2,3,
\end{align}
where, $B_{a}W$ and $C_{a}W$ are the tangential and transversal parts of $J_{a}W$.
\begin{lemma}\label{remma}
	On any screen real null submanifold $(M,g)$ of a normal indefinite complex contact manifold $\overline{M}$, we have $\overline{g}((\overline{\nabla}_{X}J_{a})\xi,PY)=0$, for all $X,Y\in \Gamma(TM)$ and $\xi \in \Gamma(\mathrm{Rad}\,TM)$.
\end{lemma}
\begin{proof}
	This follows directly from (\ref{m1})--(\ref{e21}) and  Remark \ref{rema}.
\end{proof}
\noindent Next, we prove the following result.
\begin{theorem}\label{mai2}
Let $(M,g)$ be a screen real null submanifold of a normal indefinite complex contact manifold $\overline{M}$. The induced connection $\nabla$ is a metric connection if and only if $h^{s}(X,\xi)$ has no components in $J_{a}S(TM)$, for any  $X \in \Gamma(TM)$ and $\xi\in \Gamma(\mathrm{Rad}\,TM)$	
\end{theorem}
\begin{proof}
	Using Lemma \ref{remma}, (\ref{mas10}), (\ref{mas15}) and (\ref{w20}), we derive
	\begin{align}\label{w21}
		g(\nabla_{X}J_{a}\xi-B_{a}h^{s}(X,\xi),PY)=0,
	\end{align}
	for all $X,Y\in \Gamma(TM)$. We then see from (\ref{w21}) and (\ref{mas15}) that
	\begin{align}\label{w22}
		g(A^{*}_{J_{a}\xi}X,PY)=-g(B_{a}h^{s}(X,\xi),PY).
	\end{align}
	We see that when $h^{s}(X,\xi)$ has no components in $J_{a}S(TM)$, the right hand side of (\ref{w22}) vanishes and we end up with $g(A^{*}_{J_{a}\xi}X,PY)=0$. Thus, by the fact that $S(TM)$ is non-degenerate, we get $A^{*}_{J_{a}\xi}X=0$, for any $X \in \Gamma(TM)$ and $\xi\in \Gamma(\mathrm{Rad}\,TM)$, and the result follows from Theorem \ref{th1}, which completes the proof.
\end{proof}
\noindent A null submanifold $(M,g)$ of a semi-Riemannian manifold $\overline{M}$ is called irrotational \cite[Definition 4.4.7, p. 182]{ds2} if $\overline{\nabla}_{X}\xi\in \Gamma(TM)$, for all $X\in \Gamma(TM)$ and $\xi \in \Gamma(\mathrm{Rad}\, TM)$. Note that, for a null submanifold, this is equivalent to $h^{l}(X,\xi)=h^{s}(X,\xi)=0$, for all $X\in \Gamma(TM)$ and $\xi \in \Gamma(\mathrm{Rad}\,TM)$. Thus, we have the following result.
\begin{corollary}
	Any irrotational screen real null submanifold of a normal indefinite complex contact manifold carries an induced metric connection.
\end{corollary}

\begin{theorem}
	The radical distribution, $\mathrm{Rad}\, TM$, of any screen real null submanifold of a  normal indefinite complex contact manifold is integrable if and only if $B_{a}h^{s}(\xi_{1},J_{a}\xi_{2})=B_{a}h^{s}(\xi_{2},J_{a}\xi_{1})$ , for all $\xi_{1},\xi_{2}\in \Gamma(\mathrm{Rad}\,TM)$.
\end{theorem}
\begin{proof}
	Replacing $\xi$ with $J_{a}\xi_{1}$ in Lemma \ref{remma}, and using (\ref{mas10}), (\ref{cm1}) and (\ref{w20}), we derive
	\begin{align}\label{w30}
		g(\nabla_{X}\xi_{1},PY)=-g(B_{a}h^{s}(X,J_{a}\xi_{1}),PY),
	\end{align}

for all $X,Y\in\Gamma(TM)$ and $\xi_{1}\in \Gamma(\mathrm{Rad}\,TM)$. letting $X=\xi_{2}$, where $\xi_{2}\in \Gamma(\mathrm{Rad}\,TM)$, in (\ref{w30}), we get
\begin{align}\label{w31}
		g(\nabla_{\xi_{2}}\xi_{1},PY)=-g(B_{a}h^{s}(\xi_{2},J_{a}\xi_{1}),PY).
\end{align}
It follows from that
\begin{align}\nonumber
	g([\xi_{1},\xi_{2}],PY)=g(B_{a}h^{s}(\xi_{2},J_{a}\xi_{1})-B_{a}h^{s}(\xi_{1},J_{a}\xi_{2}),PY),
\end{align}
from which our result follows.
\end{proof}
\begin{theorem}
	The screen distribution, $S(TM)$, of a screen real null submanifold of a normal indefinite complex contact manifold is integrable if and only if $-A_{J_{a}X}Y+A_{J_{a}Y}X$ has no components in $\mathrm{Rad}\,TM$, for all $X,Y\in \Gamma(S(TM))$.
\end{theorem}
\begin{proof}
	Using (\ref{m1})--(\ref{e21}) and Remark \ref{rema}, we get
	\begin{align}
		\overline{g}((\overline{\nabla}_{X}J_{a})Y,N)=0,\;\; a=1,2,3,
	\end{align}
	for all $X,Y\in \Gamma(S(TM))$ and $N\in \Gamma(l\mathrm{tr}(TM))$. Then, applying (\ref{mas10}), (\ref{mas12}) and (\ref{w20}), we derive
	$\overline{g}(\nabla_{X}Y,J_{a}N)=\overline{g}(A_{J_{a}Y}X,N)$, from which we get
	\begin{align}\label{w32}
		\overline{g}([X,Y],J_{a}N)=\overline{g}(A_{J_{a}Y}X-A_{J_{a}X}Y,N),
	\end{align}
	Then, our result follows from (\ref{w32}) and the fact that $l\mathrm{tr}(TM)$ is invariant with respect to $J_{a}$, and the proof is complete.
\end{proof}
\noindent We also have the following result.
\begin{corollary}
	The screen distribution, $S(TM)$, of a screen real null submanifold of a normal indefinite complex contact manifold is integrable if and only if $A_{J_{a}X}Y=A_{J_{a}Y}X$, for all $X,Y\in \Gamma(S(TM))$.
\end{corollary}
\noindent A null submanifold $(M,g)$, of a semi-Riemannian manifold $(\overline{M},\overline{g})$ is said to be totally umbilic in $\overline{M}$ \cite{ds2} if there is a smooth transversal vector field $\mathcal{H}\in \Gamma(\mathrm{tr}(TM))$, called the transversal curvature vector of $M$ such that $h(X,Y)=\mathcal{H}g(X,Y)$, for all  $X,Y\in \Gamma(TM)$. Moreover, it is easy to see that $M$ is totally umbilic in $\overline{M}$ if and only if on each coordinate neighborhood $\mathscr{U}$ there exists smooth vector fields $H^l\in\Gamma(l\mathrm{tr}(TM))$ and $H^s\in\Gamma(S(TM^\perp))$  such that, for all $X,Y\in\Gamma(TM)$,
  \begin{align}\label{w40}
   h^l(X,Y)=H^l g(X,Y),\;\;\; h^s(X,Y)=H^s g(X,Y).
  \end{align}
  \begin{lemma}\label{lemma3}
  	On any totally umbilic screen real null submanifold of a normal indefinite complex contact manifold, we have $\overline{g}(H^{s},U)=\overline{g}(H^{s},V)=0$.
 \end{lemma}
 \begin{proof}
 	From (\ref{mas12}) and (\ref{cm4}), we have
 	\begin{align}\label{oo23}
 		-GX+\sigma(X)V=-A_{U}X+\nabla_{X}^{s}U+D^{l}(X,U),
 	\end{align}
for any $X\in \Gamma(TM)$. Taking the product of (\ref{oo23}) with $PY$, where $Y\in \Gamma(TM)$, we get $\overline{g}(A_{U}PX,PY)=0$. It then follows from (\ref{mas13}) that $\overline{g}(h^{s}(PX,PY),U)=0$. Since $M$ is totally umbilic, we have $g(PX,PY)\overline{g}(H^{s},U)=0$. That is, $\overline{g}(H^{s},U)=0$ since $S(TM)$ is nondegenerate. Finally, $\overline{g}(H^{s},V)=0$ follow from (\ref{cm5}) and similar steps as before.
 \end{proof}
 \noindent For a totally umbilic null submanifold $M$, we see from (\ref{w40}) that $h^{l}(X,\xi)=h^{s}(X,\xi)=0$, for all $X\in \Gamma(TM)$ and $\xi\in \Gamma(\mathrm{Rad}\,TM)$. Thus, using this fact and Theorem \ref{mai2}, we state the following result.
\begin{theorem}
	Any totally umbilic screen real null submanifold of a normal indefinite complex contact manifold has an induced metric connection. Moreover,
	\begin{enumerate}
		\item $H^{s}\in \Gamma(S(TM^{\perp}))$.
		\item $M$ is either totally geodesic or $\dim S(TM)=1$.
	\end{enumerate}
\end{theorem}
\begin{proof}
	The first part of theorem is obvious. However, for the two parts (1) and (2), we let $\omega$ be the complementary vector bundle to $J_{a}S(TM)$ in $S(TM^{\perp})$. Then we can write $S(TM^{\perp})=J_{a}S(TM)\perp \omega$, for all $a=1,2,3$. Since $J_{a}(S(TM))$ and $S(TM^{\perp})$ are nondegenerate, it follows that $\omega$ is also nondegenerate. Moreover, it is easy to see that $\omega$ is invariant. Then, it means that the characteristic subbundle $\mathcal{V}\subseteq\omega$. Next, from (\ref{m1})--(\ref{e21}) and Lemma \ref{lemma11}, we see that $\overline{g}((\overline{\nabla}_{PX}J_{a})PY,J_{a}W)=0$, for all $X,Y\in \Gamma(TM)$ and $W\in \Gamma(\omega)$. It then follows from (\ref{mas10}) and (\ref{mas12}) that
\begin{align}\label{las}
		\overline{g}(h^{s}(PX,PY),J_{a}W)=-\overline{g}(C_{a}h^{s}(PX,PY),J_{a}W).	
    \end{align}
 As $M$ is totally umbilic,  relation (\ref{las}) reduce to
 \begin{align}\nonumber
 	g(PX,PY)\overline{g}(H^{s},J_{a}W)=-g(PX,PY)\overline{g}(H^{s},W),
 \end{align}
  which give $\overline{g}(H^{s},J_{a}W)=-\overline{g}(H^{s},W)$, since $S(TM)$ is nondegenerate. Hence, replacing $W$ with $J_{a}W$ and using Lemma \ref{lemma3}, we get $\overline{g}(H^{s},J_{a}W)=\overline{g}(H^{s},W)$. So, $\overline{g}(H^{s},W)=0$, and hence $H^{s}\in \Gamma(S(TM^{\perp}))$, which proves (1). For (2), we first note, from (\ref{m1})--(\ref{e21}), that $\overline{g}((\overline{\nabla}_{X}J_{a})X,Y)=0$, for all $X,Y\in \Gamma(S(TM))$. It then follows from this last relation, together with (\ref{mas10}) and  (\ref{mas12}) that
  \begin{align}\label{las1}
  	\overline{g}(h^{s}(X,X),J_{a}Y)=\overline{g}(h^{s}(X,Y),J_{a}X).  \end{align}
  The rest of the proof follows easily as in \cite[Theorem 8.3.10, p. 372]{ds2}, starting from (\ref{las1}), which proves (2). This completes the proof.
\end{proof}

\section{Screen transversal anti-invariant submanifolds}\label{main33}

Using similar language as in \cite[Definition 3.2, p. 320]{sahin}, together with Lemma \ref{impo}, we define  screen transversal anti-invariant submanifolds as follows.
\begin{definition}\label{def4}
	\rm{
	Let $(M,g)$  be a null submanifold of an indefinite complex contact manifold $\overline{M}$, which is transversal to the characteristic subbundle $\mathcal{V}$, that is $\mathcal{V}\subset S(TM^{\perp})$. Then, $M$ is called a screen transversal anti-invariant  null submanifold if $J_{a}\mathrm{Rad}\,TM\subset S(TM^{\perp})$ and $J_{a}S(TM)\subset S(TM^{\perp})$, for all $a=1,2,3$.
	}
\end{definition}
\noindent It then follows from Definition \ref{def4} that $J_{a}l\mathrm{tr}(TM)\subset S(TM^{\perp})$. Moreover, $S(TM^{\perp})$ decomposes as follows;
\begin{align}\nonumber
	S(TM^{\perp})=\{ J_{a}\mathrm{Rad}\,TM\oplus J_{a}l \mathrm{tr}(TM)\} \perp J_{a}S(TM)\perp D_{0},
\end{align}
where $D_{0}$ is a nondegenerate and $J_{a}$-invariant distribution.  Then,  for any $W\in \Gamma(S(TM^{\perp}))$, we have
\begin{align}\label{u12}
	J_{a}W=B_{a}W+C_{a}W,
\end{align}
where $B_{a}W$ and $C_{a}W$ are the tangemtial and transversal parts of $J_{a}W$.
\begin{theorem}
	The induced connection $\nabla$ on a screen transversal anti-invariant null submanifold of a normal indefinite complex contact manifold is a metric connection if and only if  $\nabla_{X}^{s}J_{a}\xi$ has no components in $J_{a}(S(TM))$, for all $X\in \Gamma(TM)$ and $\xi \in \Gamma(\mathrm{Rad}\, TM)$.
\end{theorem}
\begin{proof}
	First note that, on any screen transversal anti-invariant null submanifold $M$, the 1-forms $u$ and $v$ vanishes on $TM$. Thus, letting $Y=J_{a}\xi$ in (\ref{m1})-(\ref{e21}), we get
	\begin{align}\label{u10}
		\overline{g}((\overline{\nabla}_{X}J_{a})J_{a}\xi,PY)=0,
	\end{align}
	for all $X,Y\in \Gamma(TM)$. Applying (\ref{cm1}), (\ref{mas10}) and (\ref{mas12}) to (\ref{u10}), we get
	\begin{align}\label{u11}
		g(\nabla_{X}\xi-J_{a}A_{J_{a}\xi}X+J_{a}D^{l}(X,J_{a}\xi)+J_{a}\nabla^{s}_{X}J_{a}\xi,PY)=0.
	\end{align}
	Considering (\ref{u11}), (\ref{u12}) and (\ref{mas15}), we get
	\begin{align}\label{u13}
		g(A^{*}_{\xi}X,PY)=g(B_{a}\nabla^{s}_{X}J_{a}\xi,PY).
	\end{align}
	Finally, our result follows from (\ref{u13})  and Theorem \ref{th1}.
\end{proof}
\begin{theorem}
	The radical distribution, $\mathrm{Rad}\, TM$, of any screen transversal anti-invariant null submanifold of a normal indefinite complex contact manifold is integrable if and only if $B_{a}\nabla^{s}_{\xi_{1}}J_{a}\xi_{2}=B_{a}\nabla^{s}_{\xi_{2}}J_{a}\xi_{1}$ , for all $\xi_{1},\xi_{2}\in \Gamma(\mathrm{Rad}\,TM)$.
\end{theorem}
\begin{proof}
	From (\ref{u11}), we have
	\begin{align}\label{u15}
		g(\nabla_{X}\xi,PY)=-\overline{g}(J_{a}\nabla_{X}^{s}J_{a}\xi,PY),
	\end{align}
	for all $X,Y\in \Gamma(TM)$. Then, from (\ref{u15}) and (\ref{u12}) we get
	\begin{align}\nonumber
		g([\xi_{1},\xi_{2}],PY)=\overline{g}(B_{a}\nabla_{\xi_{2}}^{s}J_{a}\xi_{1}-B_{a}\nabla_{\xi_{1}}^{s}J_{a}\xi_{2},PY),	
	\end{align}
	form which our result follows, and the proof is completed.
\end{proof}
\begin{theorem}
	The screen distribution, $S(TM)$, of a screen transversal anti-invariant null submanifold of a normal indefinite complex contact manifold is integrable if and only if $\nabla_{X}^{s}J_{a}Y-\nabla_{Y}^{s}J_{a}X$ has no components in $J_{a}\mathrm{Rad}\,TM$, for all $X,Y\in \Gamma(S(TM))$.
\end{theorem}
\begin{proof}
	A proof follows  from (\ref{m1})-(\ref{e21}), (\ref{mas10}) and (\ref{mas12}).
\end{proof}


\end{document}